\documentclass[table]{amsart}
\usepackage[margin=1.1in]{geometry}
\usepackage{amscd,amsmath,amsxtra,amsthm,amssymb,stmaryrd,xr,mathrsfs,mathtools,enumerate,commath, comment, mathtools, orcidlink}
\usepackage{enumitem}
\usepackage{booktabs}
\usepackage{xcolor}
\usepackage{hyperref}
\usepackage{todonotes}
\definecolor{darkgoldenrod}{rgb}{0.72,0.53,0.04}
\definecolor{goldmetallic}{rgb}{0.83,0.69,0.22}
\hypersetup{
 colorlinks=true,
 linkcolor=darkgoldenrod,
 filecolor=brown,
 urlcolor=goldmetallic,
 citecolor=darkgoldenrod
}

\newtheorem{theorem}{Theorem}[section]
\newtheorem{lemma}[theorem]{Lemma}
\newtheorem{proposition}[theorem]{Proposition}

\newtheorem{question}[theorem]{Question}

\newcommand{\GL}{\operatorname{GL}}
\newcommand{\Ind}{\operatorname{Ind}}
\newcommand{\Irr}{\operatorname{Irr}}
\newcommand{\St}{\operatorname{St}}
\newcommand{\diag}{\operatorname{diag}}

\newcommand{\F}{\mathbb F}
\newcommand{\C}{\mathbb C}
\newcommand{\Q}{\mathbb Q}
\newcommand{\Z}{\mathbb Z}
\newcommand{\cO}{\mathcal O}

\newcommand{\mtx}[4]{\begin{pmatrix}#1&#2\\#3&#4\end{pmatrix}}
\newcommand{\op}[1]{\operatorname{#1}}

\numberwithin{equation}{section}

\title[Average divisibility in character tables of \(\GL_2(\F_q)\)]{Average divisibility in character tables of \(\GL_2(\F_q)\)}
\author[A.~Ray]{Anwesh Ray\, \orcidlink{0000-0001-6946-1559}}
\address[A.~Ray]{Chennai Mathematical Institute, H1, SIPCOT IT Park, Kelambakkam, Siruseri, Tamil Nadu 603103, India}
\email{anwesh@cmi.ac.in}

\author[M.~Ray]{Mishty Ray\, \orcidlink{0000-0003-0090-3966}}
\address[M.~Ray]{1984 Mathematics Rd, Vancouver, BC V6T 1Z2, Canada}
\email{mishtyray@math.ubc.ca}


\begin{document}

\begin{abstract}
Let \(q\) range over odd prime powers and let \(G_q=\GL_2(\F_q)\).  Fix a
prime number \(\ell\).  Motivated by work of Peluse and Soundararajan on
Miller's conjecture for character tables of symmetric groups, we study the
proportion of entries in the character table of \(G_q\) which are not
divisible by \(\ell\), in the sense of divisibility in the ring of algebraic
integers. We prove that
$N_\ell(q)=\frac{q^4}{2}+O_\epsilon(q^{3+\epsilon})$
for every \(\epsilon>0\), where \(N_\ell(q)\) denotes the number of entries
which are not divisible by \(\ell\). We also show that the number of zero
entries is $\frac{q^4}{2}+O_\epsilon(q^{3+\epsilon})$. Consequently, the proportion of all entries not divisible by \(\ell\) tends
to \(1/2\), while the proportion of nonzero entries not divisible by
\(\ell\) tends to \(1\).  This differs significantly from the symmetric-group case, where almost every character-table entry is divisible by any fixed prime. We also prove an angular equidistribution result for the nonzero character
values as \(q\to\infty\). We show that the arguments become equidistributed in
\([0,2\pi]\). This proves an analogue of Miller's
question on the distribution of signs among the nonzero entries in character
tables of symmetric groups.
\end{abstract}

\maketitle

\section{Introduction}
The characters of irreducible complex representations of finite groups exhibit interesting arithmetic behavior, especially when considered in infinite families. Note that the character values of finite groups are algebraic integers, in fact, they are sums of roots of unity. One may study the average arithmetic behavior of character tables for a family of finite groups and ask how often the entries satisfy a prescribed congruence or
divisibility condition. The most prominent example is the family of
symmetric groups \(S_N\). For this family, the character values are integers. The irreducible characters and conjugacy classes
of \(S_N\) are both parametrized by partitions of \(N\), and hence the
character table has \(p(N)^2\) entries, where \(p(N)\) denotes the number of partitions of $N$.
\par Miller \cite{Millerpaper} conjectured, based on extensive computations, that for
every fixed prime number \(\ell\), almost every entry in the character table of
\(S_N\) is divisible by \(\ell\) as \(N\to\infty\).  This conjecture was proved
by Peluse and Soundararajan \cite{PS1}.  They showed that
if a prime \(\ell\leq (\log N)/(\log\log N)^2\), then the number of entries in the
character table of \(S_N\) which are not divisible by \(\ell\) is $O\left(p(N)^2N^{\frac{-1}{12\ell}}\right)$. In particular, for every fixed prime \(\ell\), the proportion of entries in the
character table of \(S_N\) not divisible by \(\ell\) tends to \(0\) as $N\rightarrow \infty$. Prior to this, the conjecture was known for $\ell\leq 13$, cf. \cite{McKay,Pelusepreprint}. Subsequently, Peluse and Soundararajan \cite{PS2} proved the stronger prime-power form of Miller's conjecture. Namely, they showed that for every fixed prime power \(\ell^r\), almost every entry in the character table of \(S_N\) is divisible by \(\ell^r\) as \(N\to\infty\). Such results rely on both combinatorial methods as well as tools from analytic number theory; they underscore the growing scope of arithmetic statistics in representation theory.

\par The purpose of the present paper is to investigate analogous questions
for a different natural family of finite groups, namely $G_q=\GL_2(\F_q)$,
where \(q\) ranges over powers of odd primes. We shall see that the answer for these groups is
different from that of the family of symmetric groups $S_N$ as $N\rightarrow \infty$. We fix a prime number $\ell$. Instead of almost every
entry being divisible by $\ell$, we show that the limiting proportion
of entries not divisible by $\ell$ is exactly \(1/2\). For each odd prime power \(q\), let
\(N_\ell(q)\) denote the number of entries in the character table of
\(G_q\) which are not divisible by \(\ell\), in the sense of divisibility in
the ring of algebraic integers. Let $M_0(q)$ (resp. $M_0'(q)$) be the number of entries in the character table that are equal (resp. not equal) to $0$. The total number of entries in the character table of $G_q$ is \[M_0(q)+M_0'(q)=(q^2-1)^2.\]
We define
\[
\mathfrak d_\ell(q)
:=
\frac{N_\ell(q)}{(q^2-1)^2}\quad \text{and} \quad \mathfrak a_\ell(q):=\frac{N_\ell(q)}{M_0'(q)}
\]
and prove the following asymptotic formula.

\begin{theorem}\label{thm:main}
Let \(\ell\) be a fixed rational prime.  As \(q\to\infty\) over odd prime
powers, one has
\[
\lim_{q\to\infty}\mathfrak d_\ell(q)=\frac12\quad \text{and}\quad \lim_{q\to\infty}\mathfrak a_\ell(q)=1.
\]
That is, asymptotically one half of the entries in the character table
of \(G_q\) are not divisible by \(\ell\), and one half are divisible
by \(\ell\). Further, most nonzero entries are not divisible by $\ell$. 
\end{theorem}
More precisely, we prove the following quantitative form of Theorem
\ref{thm:main}.  For every fixed prime number $\ell$ and every
$\epsilon>0$, one has
\[
N_\ell(q)=\frac{q^4}{2}+O_\epsilon(q^{3+\epsilon})
\]
as $q\to\infty$ over odd prime powers.  Moreover, the number $M_0(q)$ of
zero entries in the character table satisfies
\[
M_0(q)=\frac{q^4}{2}+O_\epsilon(q^{3+\epsilon}).
\]
For further details, we refer to Theorem \ref{explicit estimate theorem}. Thus the main contribution to the entries divisible by $\ell$ comes from
the zero entries, while almost all nonzero entries are not divisible by
$\ell$.
\par Our result also settles the question for powers of
\(\ell\). Indeed, if a nonzero algebraic integer is not divisible by
\(\ell\), then it is not divisible by any power \(\ell^r\). Since almost all
nonzero entries in the character table of \(G_q\) are not divisible by
\(\ell\), it follows a fortiori that, for every fixed \(r\geq 1\), almost all
nonzero entries are not divisible by \(\ell^r\).  Thus the limiting proportion
of all entries not divisible by \(\ell^r\) is again \(1/2\), and the limiting
proportion of nonzero entries not divisible by \(\ell^r\) is \(1\).

\par The ideas used in the proof are quite different from those in \cite{PS1}. Rather than relying on the combinatorics of partitions, our argument uses the standard classification of irreducible complex representations of $G_q=\GL_2(\F_q)$ into four families: namely one-dimensional, principal series, Steinberg twists, and cuspidal. On the conjugacy class side, there are also four natural families, namely scalar, split regular semisimple, elliptic, and Jordan (i.e., consisting of a single Jordan block). The character table naturally decomposes into \(4\times 4\) blocks,
with rows indexed by the four families of irreducible representations and
columns indexed by the four families of conjugacy classes:
\begin{table}[h]
    \centering
    \begin{tabular}{|c|c|c|c|c|}
    \hline
        & 
\text{scalar} &
\text{split regular semisimple} &
\text{elliptic} &
\text{non-semisimple Jordan}
\\
\hline
\hline
\text{one-dimensional}
&
$\ast$ & $\ast$ & $\ast$ & $\ast$
\\
\text{principal series}
&
$\ast$ & $\ast$ & $0$ & $\ast$
\\
\text{Steinberg twists}
&
$\ast$ & $\ast$ & $\ast$ & $0$
\\
\text{cuspidal}
&
$\ast$ & $0$ & $\ast$ & $\ast$\\ \hline
    \end{tabular}
    \caption{Representations and conjugacy classes}
    \label{table: Representations and conjugacy classes}
\end{table}

\iffalse
\begin{equation}\label{equation: array 1}
\begin{array}{c|c|c|c|c}
 & 
\text{scalar} &
\text{split regular semisimple} &
\text{elliptic} &
\text{non-semisimple Jordan}
\\
\hline
\text{one-dimensional}
&
\ast & \ast & \ast & \ast
\\
\text{principal series}
&
\ast & \ast & 0 & \ast
\\
\text{Steinberg twists}
&
\ast & \ast & \ast & \ast
\\
\text{cuspidal}
&
\ast & 0 & \ast & \ast
\end{array}
\end{equation}\fi
Here \(\ast\) denotes a block whose entries are not all zero. The one-dimensional representations and the Steinberg twists form small
families, each of size \(q-1\).  Since the total number of conjugacy classes
is \(q^2-1\), the entries lying in these rows contribute only \(O(q^3)\)
entries to the full character table. Thus, their contribution has density zero. Similarly, the scalar and Jordan
conjugacy classes each form a family of size \(q-1\), and therefore the
entries lying in these columns also have density zero.
\par Consequently, the limiting density is determined entirely by the four large
blocks obtained by evaluating principal series and cuspidal representations
on split regular semisimple and elliptic conjugacy classes.  The situation is
summarized in the following table:
\begin{table}[h]
    \centering
    \begin{tabular}{|c|c|c|}
        \hline

& \text{Split regular semisimple classes} & \text{Elliptic classes} \\
\hline
\hline
\text{Principal series representations}
&
\text{nonzero in general}
&
$0$
\\         
\text{Cuspidal representations}
&
$0$
&
\text{nonzero in general}\\
\hline
    \end{tabular}
    \caption{Summary of contributions to the density}
    \label{tab:placeholder}
\end{table}

\iffalse\[
\begin{array}{c|c|c}
 & \text{Split regular semisimple classes} & \text{Elliptic classes} \\
\hline
\text{Principal series representations}
&
\text{nonzero in general}
&
0
\\[4pt]
\text{Cuspidal representations}
&
0
&
\text{nonzero in general}
\end{array}
\]\fi

The $0$-blocks contribute to the density of entries divisible by \(\ell\).
The remaining two large blocks are the principal-series/split block and the
cuspidal/elliptic block.  Each of these blocks contributes a density of $\frac{1}{4}$ to the proportion of character entries that are not divisible by $\ell$.

\par Miller \cite[Question~3]{Millerzeros} also considered an archimedean analogue of the divisibility
question for the symmetric groups. Since the character values of \(S_N\)
are integers, one may ask how often the nonzero entries in the character
table of \(S_N\) are positive or negative. For
\(G_q=\GL_2(\F_q)\), the character values are in general complex rather than
real, so the corresponding question is not a sign-distribution question but
an angular-distribution question. We study how the arguments of the nonzero character values are distributed
in \([0,2\pi]\).
\par Let $\op{Irr}(G_q)$ be the set of isomorphism classes of complex irreducible representations of $G_q$. For a nonzero complex number \(z\), write $\operatorname{arg}(z)\in [0, 2\pi)$ for its argument. If \(I=[a,b]\subset [0,2\pi)\) is an interval, let
\[
P_q(I)
:=
\frac{
\#\{(\pi,C): \pi\in \Irr(G_q),\ C\text{ a conjugacy class of }G_q,\ 
\chi_\pi(C)\neq 0,\ \operatorname{arg}(\chi_\pi(C))\in I\}
}{
\#\{(\pi,C): \pi\in \Irr(G_q),\ C\text{ a conjugacy class of }G_q,\ 
\chi_\pi(C)\neq 0\}
}.
\]
Thus \(P_q(I)\) is the proportion of nonzero entries in the character table
whose argument lies in \(I\).

\begin{theorem}\label{thm:argument-equidistribution}
Let \(I=[a,b]\subset [0,2\pi]\).  As \(q\to\infty\) over odd prime powers,
\[
\lim_{q\to\infty} P_q(I)=\frac{b-a}{2\pi}.
\]
\end{theorem}
\noindent For related work on character values of finite general linear groups, we
mention the work of Gurevich and Howe \cite{GurevichHowe}. The authors study asymptotics for the normalized trace $\frac{\op{tr}\pi(g)}{\dim \pi}$ via a harmonic-analytic approach.

\par We conclude by mentioning some natural further directions. The results of this paper suggest that the statistics of character tables can vary substantially with the family of finite groups under consideration. It would be interesting to study the analogous questions for other families of finite groups of Lie type. A first natural case is $\GL_n(\F_q)$, either with $n$ fixed and $q\to\infty$, or with $q$ fixed and $n\to\infty$.  The irreducible characters of $\GL_n(\F_q)$ were parametrized by Green \cite{Green}. More generally, Deligne--Lusztig theory \cite{DL76} provides a uniform framework for studying irreducible characters of finite reductive groups. We pose the following question for a fixed algebraic group and varying finite field.
\begin{question}\label{question:finite-lie-type-divisibility}
Let \(\mathcal G\) be a reductive group scheme over \(\mathbb Z\), and
write $G_q:=\mathcal G(\F_q)$
for its group of \(\F_q\)-points. 
\begin{enumerate}
    \item Fix a prime number \(\ell\). What is the
limiting proportion, as \(q\to\infty\) over prime powers, of entries in the
character table of \(G_q\) which are not divisible by \(\ell\)?  More
generally, what is the limiting proportion of nonzero entries which are not
divisible by \(\ell\)?
\item Is there a measure on $[0,2\pi]$ with respect to which the arguments of the non-zero character values of irreducible representations become equidistributed as $q\rightarrow \infty$?
\end{enumerate}
\end{question}

We also pose the analogous question in fixed characteristic, where the rank
of the group tends to infinity.

\begin{question}\label{question:fixed-field-growing-rank}
Let \(\F_q\) be a fixed finite field, and let \(\{G_n\}_{n\geq 1}\) be one
of the standard families of finite classical groups, for instance $\GL_n(\F_q)$, $\op{SL}_n(\F_q)$, $\mathrm U_n(\F_q)$, $\op{SO}_n(\F_q)$ or $\op{Sp}_{2n}(\F_q)$.
\begin{enumerate}
    \item Fix a prime number \(\ell\).  What is the limiting proportion, as
    \(n\to\infty\), of entries in the character table of \(G_n\) which are
    not divisible by \(\ell\)?  More generally, what is the limiting
    proportion of nonzero entries which are not divisible by \(\ell\)?

    \item After discarding the zero entries, is there a natural probability
    measure on \([0,2\pi]\) with respect to which the arguments of the
    remaining character values become equidistributed as \(n\to\infty\)?
\end{enumerate}
\end{question}

\subsection*{Acknowledgments}

We thank Mathilde Gauthier, Sarah Peluse, Dipendra Prasad and Steven Spallone for helpful feedback on our paper.  We are especially
grateful to Sarah Peluse for suggesting that we study the equidistribution
of arguments of nonzero character values, and for pointing us to Miller's
question on the distribution of signs among the nonzero entries in character
tables of symmetric groups.

\section{Preliminaries}

\subsection{Conjugacy classes in \(\GL_2(\F_q)\)}\label{s 2.1}

Throughout the paper \(q\) is a power of an odd prime number and \(G_q=\GL_2(\F_q)\).  We write \(B\), \(T\), \(N\), and \(Z\) for the standard Borel subgroup, diagonal torus, unipotent subgroup, and center, respectively. In particular, \(B\) consists of the upper triangular matrices in \(G_q\), \(T\) of the diagonal matrices, \(N\) of the matrices \(\mtx{1}{x}{0}{1}\), with \(x\in\F_q\), and \(Z\) of the scalar matrices.

The conjugacy classes of \(G_q\) are divided into four families.  First, there are the scalar classes \(aI\), where \(a\in\F_q^\times\).  Second, there are the split regular semisimple classes, represented by \(\diag(a,b)\), where \(a,b\in\F_q^\times\) and \(a\neq b\).  These are parametrized by unordered pairs \(\{a,b\}\).  Third, there are the elliptic classes, namely the regular semisimple elements whose characteristic polynomial is irreducible over \(\F_q\).  If \(\F_{q^2}\) is viewed as a two-dimensional vector space over \(\F_q\), multiplication gives an embedding \begin{equation}\label{defnofiota}\iota:\F_{q^2}^\times\hookrightarrow \op{Aut}_{\F_q}(\F_{q^2})\xrightarrow{\sim}G_q.\end{equation}  The elliptic classes are represented by \(\iota(\alpha)\), with \(\alpha\in\F_{q^2}^\times\setminus \F_q^\times\).  Finally, there are the non-semisimple Jordan classes represented by \(\mtx{a}{1}{0}{a}\), with \(a\in\F_q^\times\).
The total number of conjugacy classes equals
\[
(q-1)+\frac{(q-1)(q-2)}2+\frac{q^2-q}{2}+(q-1)=q^2-1.
\]
This is summarized in the table below.
\begin{table}[h]
\centering
\begin{tabular}{|c|c|c|c|}
\hline
Type & Representative & Description & Number of classes \\
\hline
\hline
Scalar & $\begin{psmallmatrix} a & 0 \\ 0 & a \end{psmallmatrix}$ & Central elements & $q-1$ \\
Split semisimple & $\begin{psmallmatrix} a & 0 \\ 0 & b \end{psmallmatrix}$, $a\neq b$ & Diagonalizable over $\F_q$ & $\frac{(q-1)(q-2)}{2}$ \\
Elliptic & $\iota(\alpha),\,\alpha\in \F_{q^2}^\times\backslash \F_q^\times$ & Irreducible characteristic polynomial & $\frac{q^2-q}{2}$ \\
Jordan & $\begin{psmallmatrix} a & 1 \\ 0 & a \end{psmallmatrix}$ & Non-semisimple & $q-1$\\
\hline
\end{tabular}
\caption{Number of conjugacy classes}
    \label{table: number of conjugacy classes}
\end{table}
\subsection{Irreducible representations of \(G_q\)}\label{s 2.2}
\par Given a representation 
\[\rho: G_q\rightarrow \op{GL}(V),\] the character of $\rho$ is defined as follows:
\[\chi(g):=\op{trace}\rho(g)\]for $g\in G_q$. In this subsection, we recall the standard classification of the irreducible complex representations of \(G_q\) following \cite[Ch.2]{BH06}.
\par The first family consists of one-dimensional representations. Such representations are of the form \(\chi\circ\det\), where \(\chi:\F_q^\times\to\C^\times\) is a character. There are \(q-1\) one-dimensional representations.

The second family consists of principal series representations.  Let \(\chi_1,\chi_2:\F_q^\times\to\C^\times\) be characters.  We regard \(\chi_1\otimes\chi_2\) as a character of \(B\) by sending \(\mtx{a}{b}{0}{d}\) to \(\chi_1(a)\chi_2(d)\).  If \(\chi_1\neq\chi_2\), the induced representation \(\Ind_B^{G_q}(\chi_1\otimes\chi_2)\) is irreducible.  Moreover, interchanging \(\chi_1\) and \(\chi_2\) gives an isomorphic representation, and these are the only identifications.  Therefore, the number of principal series representations is \((q-1)(q-2)/2\).

The third family consists of twists of the Steinberg representation. The representation $\Ind_B^{G_q}(1)$ decomposes as a direct sum $1\oplus \op{St}$, where $\op{St}$ is the Steinberg representation. We note that this is also the permutation representation of \(G_q\) on \(G_q/B\).  Twisting the $q$ dimensional representation $\op{St}$ by \(\chi\circ\det\), where \(\chi\) ranges over the characters of \(\F_q^\times\), gives \(q-1\) irreducible representations.
\par The fourth family consists of cuspidal representations, i.e., the irreducible representations not containing the trivial character of $N$.  We recall the
explicit construction following \cite{BH06}.  Let
\(E=\F_{q^2}\), viewed as a two-dimensional vector space over \(\F_q\).
Then \(E^\times\) embeds in \(G_q\) via the map $\iota$ (see \eqref{defnofiota}). Fix a nontrivial character
\(\psi:N\to \C^\times\).  If
\(\theta:E^\times\to \C^\times\) is a character, then its restriction to
\(\F_q^\times\subset E^\times\) gives a character of \(Z\).  Hence, we obtain
a character \(\theta\otimes\psi\) of \(ZN\) by
\[
(\theta\otimes\psi)(zu)=\theta(z)\psi(u),
\quad\text{for}\quad z\in Z \quad \text{and}\quad u\in N,
\]
where \(Z\simeq \F_q^\times\) via \(aI\mapsto a\).

A character \(\theta:E^\times\to\C^\times\) is called regular if
\(\theta\neq\theta^q\), where \(\theta^q(\alpha):=\theta(\alpha^q)\).
Consider the virtual representation
\[
\pi_\theta
:=
\Ind_{ZN}^{G_q}(\theta|_{\F_q^\times}\otimes\psi)
-
\Ind_{E^\times}^{G_q}\theta.
\]
where $\theta\vert_{\mathbb{F}_q^\times}$ is interpreted as a character of $Z$. Then \(\pi_\theta\) is in fact a cuspidal representation of
\(G_q\) of dimension \(q-1\), cf. \cite[p.~47]{BH06}.  Moreover,
\[
\pi_\theta\simeq \pi_{\theta'}
\quad\text{if and only if}\quad
\theta'=\theta \text{ or } \theta'=\theta^q.
\]
 Every cuspidal representation of \(G_q\) is obtained in this way.

The non-regular characters are precisely those fixed by Frobenius.  These
are exactly the characters which factor through the norm map
\(N_{E/\F_q}:E^\times\to\F_q^\times\), and hence there are \(q-1\) of them.
Since \(E^\times\) has \(q^2-1\) characters in total, there are
\(q^2-1-(q-1)=q(q-1)\) regular characters.  Dividing by the Frobenius
equivalence \(\theta\sim\theta^q\), we obtain \(q(q-1)/2\) cuspidal
representations.

\subsection{Divisibility of character values}

Let \(K\) be a number field containing the values of the characters under consideration, and let \(\cO_K\) be its ring of integers.  If \(\alpha\in\cO_K\), we say that \(\ell\) divides \(\alpha\) if \(\alpha\in \ell\cO_K\), equivalently if \(\alpha/\ell\in\cO_K\).  This notion is independent of enlarging \(K\).  In particular, since \(0/\ell=0\), the value \(0\) is divisible by every prime \(\ell\).

We shall repeatedly use the following elementary observations about sums of roots of unity. Let $\Phi_n(x)$ denote the $n$-th cyclotomic polynomial.
\begin{lemma}\label{basic lemma 1}
For \(n\geq 1\), one has
\[
\Phi_n(1)
=
\begin{cases}
0, & n=1,\\
p, & n=p^r \text{ for some prime }p\text{ and some }r\geq 1,\\
1, & n>1 \text{ and } n \text{ is not a prime power}.
\end{cases}
\]
\end{lemma}

\begin{proof}
For \(n=1\), we have \(\Phi_1(x)=x-1\), and hence
\[
\Phi_1(1)=0.
\]
\noindent For $n>1$, we have that
\[
\frac{x^n-1}{x-1}
=
\prod_{\substack{d\mid n\\ d>1}}\Phi_d(x).
\]
Evaluating at \(x=1\), we get
\[
n=\prod_{\substack{d\mid n\\ d>1}}\Phi_d(1).
\]

We prove the desired formula by induction on \(n\). We have $\Phi_2(x)=x+1$ so the case \(n=2\) is immediate. Assume that the formula has been proved for all integers \(m\) with
\(1<m<n\). We prove it for \(n\). First suppose that \(n=p^r\) is a prime power. We find that
\[
p^r=\prod_{j=1}^{r}\Phi_{p^j}(1).
\]
By the induction hypothesis, for \(1\leq j<r\), we have
\[
\Phi_{p^j}(1)=p.
\]
Therefore,
\[
p^r
=
\left(\prod_{j=1}^{r-1}\Phi_{p^j}(1)\right)\Phi_{p^r}(1)
=
p^{r-1}\Phi_{p^r}(1).
\]
Hence,
\[
\Phi_{p^r}(1)=p.
\]

Now suppose that \(n>1\) is not a prime power. Write its prime factorization as
\[
n=\prod_{i=1}^{s}p_i^{a_i},
\]
where \(s\geq 2\). 
%Applying \((1)\), we have\[n=\prod_{\substack{d\mid n\\ d>1}}\Phi_d(1).\]
By the induction hypothesis, for every proper divisor \(d<n\),
\[
\Phi_d(1)
=
\begin{cases}
p_i, & d=p_i^j,\\
1, & d>1 \text{ and } d \text{ is not a prime power}.
\end{cases}
\]
Since \(n\) itself is not a prime power, every prime-power divisor of \(n\) is
proper. Therefore
\[
\prod_{\substack{d\mid n\\ 1<d<n}}\Phi_d(1)
=
\prod_{i=1}^{s}\prod_{j=1}^{a_i} p_i
=
\prod_{i=1}^{s}p_i^{a_i}
=
n.
\]
Thus,
\[
n
=
\prod_{\substack{d\mid n\\ d>1}}\Phi_d(1)
=
\left(\prod_{\substack{d\mid n\\  1<d<n}}\Phi_d(1)\right)
\Phi_n(1)
=
n\Phi_n(1).
\]
Hence $\Phi_n(1)=1$ when $n$ is not a prime power, as required.
\end{proof}

\begin{lemma}\label{lem:root-unity}
Let \(\zeta\) be a root of unity and let \(\ell\) be a rational prime.
Then \((1+\zeta)/\ell\) is an algebraic integer if and only if either
\(\zeta=-1\), or \(\ell=2\) and \(\zeta=1\).
\end{lemma}

\begin{proof}
If \(\zeta=-1\), then \(1+\zeta=0\), and hence
\((1+\zeta)/\ell=0\) is an algebraic integer.  If \(\zeta=1\) and
\(\ell=2\), then \((1+\zeta)/\ell=1\), and this is again an algebraic
integer.

Conversely, suppose that \((1+\zeta)/\ell\) is an algebraic integer.  Let
\(m\) be the order of \(\zeta\).  We work in the cyclotomic field
\(\Q(\zeta_m)\), where \(\zeta_m\) is a primitive \(m\)-th root of unity,
and let \(\Phi_m\) be the \(m\)-th cyclotomic polynomial.  Replacing
\(\zeta\) by \(\zeta_m\), the assumption implies that \(1+\zeta_m\) is
divisible by \(\ell\) as an algebraic integer.  Therefore $\ell^{\varphi(m)}$ divides $\left|N_{\Q(\zeta_m)/\Q}(1+\zeta_m)\right|$. We observe that
\[
\left|N_{\Q(\zeta_m)/\Q}(1+\zeta_m)\right|
=
|\Phi_m(-1)|.
\]
According to Lemma \ref{basic lemma 1}, we have that
\[
\Phi_n(1)
=
\begin{cases}
0, & n=1,\\
p, & n=p^r \text{ for some prime }p\text{ and some }r\geq 1,\\
1, & n>1 \text{ and } n \text{ is not a prime power}.
\end{cases}
\]
When \(m=2n\) with \(n\) odd, the identity
\[\Phi_{2n}(X)=\Phi_n(-X)
\]
implies that
\[
|\Phi_m(-1)|=|\Phi_n(1)|.
\]
From this one finds that
\[
|\Phi_m(-1)|
=
\begin{cases}
0, & m=2,\\
2, & m=2^r \text{ for some } r\geq 2,\\
p, & m=2p^r \text{ for some odd prime }p\text{ and some }r\geq 1,\\
1, & \text{otherwise},
\end{cases}
\]
for \(m>1\). Assuming that $m>2$, this
contradicts the divisibility
\[
\ell^{\varphi(m)}
\mid
\left|N_{\Q(\zeta_m)/\Q}(1+\zeta_m)\right|.
\]
Therefore the only possible cases are \(m=1\), with \(\ell=2\), and
\(m=2\).  Equivalently, either \(\zeta=1\) and \(\ell=2\), or
\(\zeta=-1\).
\end{proof}

\section{The Character values}

The first step is to compute $\chi(C)$ in all $16$ cases, where $C$ ranges over the four types of conjugacy classes in $G_q$ from section \ref{s 2.1} and $\chi$ ranges over the four types of irreducible representations of $G_q$ from section \ref{s 2.2}. After computing them, we will be able to prove asymptotic formulas for $N_\ell(q)$. 

\par If $H \leq G$ is a subgroup and $\sigma$ is a class function on $H$, then the character of the induced representation $\Ind_H^G \sigma$ is given by
\begin{equation}\label{char-formula}\chi_{\Ind_H^G \sigma}(g)
= \frac{1}{|H|} \sum_{\substack{x \in G \\ x^{-1} g x \in H}} \sigma(x^{-1} g x);
\end{equation}
see \cite[Theorem 12, p.~30]{Ser77}.
We begin by considering the characters of the one-dimensional representations.

\begin{lemma}[One-dimensional representations]
Let $\chi = \varepsilon \circ \det$ be a one-dimensional representation of $G_q$.
Then we have that
\begin{align*}
\chi(aI) &= \varepsilon(a^2), \\
\chi\left(\mtx{a}{0}{0}{b}\right) &= \varepsilon(ab), \quad a \neq b, \\
\chi\left(\mtx{a}{1}{0}{a}\right)&= \varepsilon(a^2), \\
\chi\left(\iota(\alpha)\right) &=\varepsilon(\alpha^{q+1})\quad \text{for}\quad  \alpha\in \F_{q^2}\setminus \F_q,
\end{align*}
where $\iota: \F_{q^2}^\times \hookrightarrow G_q$ is the natural inclusion from \eqref{defnofiota}.
\end{lemma}

\begin{proof}
The first three assertions are immediate. We note that for $\alpha\in \F_{q^2}\setminus \F_q$, \[\det \iota(\alpha) = \op{Norm}_{\F_{q^2}/\F_q}(\alpha) = \alpha^{q+1}.\]
Thus the last assertion follows.
\end{proof}
Next we compute the character values of principal series representations.
\begin{lemma}[Principal series representations]\label{princseriescalc}
Let $\pi = \Ind_B^G(\chi_1 \otimes \chi_2)$ with $\chi_1 \neq \chi_2$, and let $\chi$ be its character. Then:
\[
\chi(g) =
\begin{cases}
(q+1)\chi_1(a)\chi_2(a), & g = aI,\\[6pt]
\chi_1(a)\chi_2(b) + \chi_1(b)\chi_2(a), & g \sim \diag(a,b),\ a \neq b,\\[6pt]
\chi_1(a)\chi_2(a), & g \sim \begin{psmallmatrix} a & 1 \\ 0 & a \end{psmallmatrix},\\[6pt]
0, & g \sim \iota(\alpha),\quad \alpha\in \F_{q^2}\setminus \F_q.
\end{cases}
\]
\end{lemma}

\begin{proof}
We use the induced character formula \eqref{char-formula} to find that
\[
\chi(g) = \frac{1}{|B|} \sum_{\substack{x \in G_q \\ x^{-1} g x \in B}} (\chi_1 \otimes \chi_2)(x^{-1} g x).
\]
\par If $g = aI$, then every conjugate lies in $B$, and the value is simply
\[
\chi(g) = \frac{|G_q|}{|B|} \chi_1(a)\chi_2(a) = (q+1)\chi_1(a)\chi_2(a),
\]
since $|G_q/B| =|\mathbb P^1(\F_q)|= q+1$.
\par If $g$ is split semisimple, say $g \sim \diag(a,b)$ with $a \neq b$, then $g$ is conjugate into $B$ by $x\in B\sqcup Bw$ where $w=\mtx{0}{1}{1}{0}$. The contribution of all $x\in B$ (resp. $x\in Bw$) is $\chi_1(a)\chi_2(b)$ (resp. $\chi_1(b)\chi_2(a)$). This gives the stated formula.
%$x\in B\sqcup Bw$ by direct computation
\par If $g$ is a Jordan block, then it lies in $B$ and has a unique eigenvalue $a$. Any eigenvector of $g$ must lie in the line spanned by $e_1$. Suppose that $x^{-1} g x\in B$. Then since $x^{-1} g x$ is not scalar and also has the same eigenvalues as $g$, we may write 
\[x^{-1} g x=\mtx{a}{\ast}{0}{a},\] where $\ast\neq 0$. Thus any eigenvector of $x^{-1}g x$ must also belong to the line spanned by $e_1$. It follows that $x$ stabilizes the line spanned by $e_1$, i.e., $x\in B$. We deduce that \[\chi(g)=\chi_1(a)\chi_2(a).\]
\par Lastly, if $g\sim \iota(\alpha)$ for $\alpha\in \F_{q^2}\setminus \F_q$, then it is not conjugate into $B$, since $B$ consists of matrices whose characteristic polynomial splits over $\F_q$. Hence $\chi(g)=0$.
\end{proof}

\begin{lemma}[Steinberg representations]
Let $\pi = \St$, the Steinberg representation. The character $\chi$ of $\pi$ is given as follows.
\[
\chi(g) =
\begin{cases}
q, & g = aI,\\[6pt]
1, & g \sim \diag(a,b),\ a \neq b,\\[6pt]
0, & g \sim \begin{psmallmatrix} a & 1 \\ 0 & a \end{psmallmatrix},\\[6pt]
-1, & g \sim \iota(\alpha),\quad \alpha\in \F_{q^2}\setminus \F_q.
\end{cases}
\]
\end{lemma}

\begin{proof}
Recall from Section \ref{s 2.2} that $\Ind_{B}^{G_q}(1)=\St\oplus 1.$
Since trace is additive, we have the identity
\[\chi=\chi_{\Ind_B^{G_q}(1)}-1.\]
The computation of $\chi_{\Ind_B^{G_q}(1)}$ is the same as in Lemma \ref{princseriescalc} (where the assumption $\chi_1\neq \chi_2$ has not been used in the proof). The result follows from this.
\end{proof}
Next we consider characters of cuspidal representations.

\begin{lemma}[Cuspidal representations]\label{cuspcalclemma}
Let
\(\theta:E^\times\to \C^\times\) be a regular character and let \(\pi_\theta\) be the corresponding cuspidal representation as defined in section \ref{s 2.2}. Let \(\chi\) denote its character. Then
\[
\chi(g)
=
\begin{cases}
(q-1)\theta(a), &
g=aI,\\[6pt]
0, &
g\sim \diag(a,b),\ a\neq b,\\[6pt]
-\theta(a), &
g\sim
\begin{psmallmatrix}
a & 1\\
0 & a
\end{psmallmatrix},\\[12pt]
-\theta(\alpha)-\theta(\alpha^q), &
g\sim \iota(\alpha),\quad
\alpha\in \F_{q^2}\setminus \F_q.
\end{cases}
\]
\end{lemma}

\begin{proof}
Recall from section \ref{s 2.2} that we fix a nontrivial
character
\[
\psi:N\longrightarrow \C^\times.
\]
After identifying \(N\) with the additive group of \(\F_q\)
by
\[
x\longmapsto
u_x:=
\begin{pmatrix}
1 & x\\
0 & 1
\end{pmatrix},
\]
we interpret $\psi$ as an additive character $\psi:\F_q\longrightarrow \C^\times$. We identify \(Z\) with \(\F_q^\times\) by \(a\mapsto aI\).  Since
\(\F_q^\times\subset E^\times\), the restriction of \(\theta\) to
\(\F_q^\times\) gives a character of \(Z\).  We define a character $\eta:ZN\longrightarrow \C^\times$ by
\begin{equation}\label{etadefn}\eta(aI\,u_x):=\theta(a)\psi(x),\end{equation}
for $a\in\F_q^\times$ and $x\in\F_q$.
The standard construction gives the cuspidal representation attached to
\(\theta\) as the virtual representation
\[
\pi
=
\Ind_{ZN}^{G}\eta
-
\Ind_{E^\times}^{G}\theta.
\]
Thus we have
\[
\chi(g)
=
\chi_{\Ind_{ZN}^{G}\eta}(g)
-
\chi_{\Ind_{E^\times}^{G}\theta}(g).
\]

We now compute this difference for each of the four types of conjugacy
classes in \(G_q\). First suppose that \(g=aI\) is scalar, with \(a\in\F_q^\times\).  Since
\(aI\) is central, we have \(x^{-1}gx=aI\) for every \(x\in G_q\).  Hence
\[
\chi_{\Ind_{ZN}^{G_q}\eta}(aI)
=
\frac{1}{|ZN|}
\sum_{x\in G_q}\eta(aI)
=
\frac{|G_q|}{|ZN|}\theta(a).
\]
Now
\[
\frac{|G_q|}{|ZN|}
=
q^2-1,
\]
and so
\[
\chi_{\Ind_{ZN}^{G}\eta}(aI)
=
(q^2-1)\theta(a).
\]
Similarly,
\[
\chi_{\Ind_{E^\times}^{G_q}\theta}(aI)
=
\frac{1}{|E^\times|}
\sum_{x\in G_q}\theta(a)
=
\frac{|G_q|}{|E^\times|}\theta(a).
\]
Since \(|E^\times|=q^2-1\), we get
\[
\frac{|G_q|}{|E^\times|}
=
q(q-1).
\]
Therefore
\[
\chi_{\Ind_{E^\times}^{G}\theta}(aI)
=
q(q-1)\theta(a).
\]
Subtracting, we obtain
\[
\chi(aI)
=
(q^2-1)\theta(a)-q(q-1)\theta(a)
=
(q-1)\theta(a).
\]

Next suppose that
\[
g\sim \diag(a,b),\quad \text{where}\quad
a,b\in\F_q^\times,\quad \text{with} \quad a\neq b.
\]
We show that both induced characters vanish. Every
element of \(ZN\) has the form
\[
aI\,u_x
=
\begin{pmatrix}
a & ax\\
0 & a
\end{pmatrix}.
\]
Such an element has only one eigenvalue, namely \(a\).  In particular,
no element of \(ZN\) is conjugate to a split regular semisimple element
with two distinct eigenvalues \(a\) and \(b\).  Hence there is no
\(x\in G_q\) such that \(x^{-1}gx\in ZN\), and therefore \[\chi_{\Ind_{ZN}^{G}\eta}(g)=0.\]On the other hand, every element of \(E^\times\) is either scalar
or has irreducible characteristic polynomial over \(\F_q\).  Indeed, if
\(\alpha\in E^\times\setminus \F_q^\times\), then multiplication by
\(\alpha\) has characteristic polynomial equal to the minimal polynomial
of \(\alpha\) over \(\F_q\), which is irreducible of degree \(2\).  Thus
no element of \(E^\times\) is conjugate to \(\diag(a,b)\) with
\(a\neq b\).  Hence
\[
\chi_{\Ind_{E^\times}^{G}\theta}(g)=0
\]
and consequently
\[
\chi(g)=0.
\]

Now suppose that \(g\) is a non-semisimple Jordan block, with eigenvalue
\(a\in\F_q^\times\). We write
\[
g=
\begin{pmatrix}
a & 1\\
0 & a
\end{pmatrix}
=
aI\,u_{a^{-1}}.
\]
It is more convenient to compute with $g=aI\,u_1$. This gives the same conjugacy class, since all nontrivial unipotent
Jordan blocks with eigenvalue \(a\) are conjugate.

First consider the contribution from \(E^\times\). Every element of
\(E^\times\) is semisimple as an element of \(G_q\), because \(E^\times\)
is contained in a torus.  But \(g=aIu_1\) is not semisimple.  Hence \(g\)
is not conjugate to any element of \(E^\times\), and therefore
\[
\chi_{\Ind_{E^\times}^{G}\theta}(g)=0.
\]

It remains to compute $\chi_{\Ind_{ZN}^{G}\eta}(g)$. By the induced character formula,
\[
\chi_{\Ind_{ZN}^{G}\eta}(g)
=
\frac{1}{|ZN|}
\sum_{\substack{x\in G_q\\ x^{-1}gx\in ZN}}
\eta(x^{-1}gx).
\]
Suppose that $x\in G_q$ satisfies $x^{-1}gx\in ZN$. Then from \eqref{etadefn}, we have that
\[\eta(x^{-1}g x)=\eta(aIx^{-1}u_1x)=\theta(a)\eta(x^{-1}u_1x).\]
Since $x^{-1} u_1 x\in Z N$ and both its eigenvalues are equal to $1$, we have that $x^{-1} u_1 x\in N$. Therefore, 
\[\theta(a)\eta(x^{-1}u_1x)=\theta(a)\psi(x^{-1}u_1x).\]
Thus we find that 
\[
\chi_{\Ind_{ZN}^{G_q}\eta}(g)
=\frac{\theta(a)}{|ZN|}
\sum_{\substack{x\in G\\ x^{-1}u_1x\in N}}
\psi(x^{-1}u_1x).
\]
It is easy to see that $x^{-1}u_1x\in N$ if and only if $x\in B$, the Borel subgroup of $G_q$. Writing
$x=
\begin{pmatrix}
r & s\\
0 & t
\end{pmatrix}$, a direct computation gives
\[
x^{-1}u_1x
=
u_{t/r}
\]
and therefore,
\[
\chi_{\Ind_{ZN}^{G}\eta}(g)
=
\frac{\theta(a)}{|ZN|}
\sum_{\substack{r,t\in\F_q^\times\\ s\in\F_q}}
\psi(t/r).
\]
The summand does not depend on \(s\), so the sum over \(s\) contributes a
factor of \(q\). We obtain
\[
\chi_{\Ind_{ZN}^{G}\eta}(g)
=
\frac{\theta(a)}{(q-1)}
\sum_{r,t\in\F_q^\times}\psi(t/r).
\] For fixed \(r\), the quotient \(t/r\) runs once
through \(\F_q^\times\) as \(t\) runs through \(\F_q^\times\).  Therefore
\[
\sum_{t\in\F_q^\times}\psi(t/r)
=
\sum_{u\in\F_q^\times}\psi(u).
\]
Since \(\psi\) is a nontrivial additive character of \(\F_q\), we have
\[
\sum_{u\in\F_q}\psi(u)=0,
\]
and hence
\[
\sum_{u\in\F_q^\times}\psi(u)
=
-1.
\]
It follows that
\[
\sum_{r,t\in\F_q^\times}\psi(t/r)
=
\sum_{r\in\F_q^\times}(-1)
=
-(q-1).
\]
Substituting this into the induced character formula gives
\[
\chi_{\Ind_{ZN}^{G}\eta}(g)
=
\frac{\theta(a)}{(q-1)}\cdot (-(q-1))
=
-\theta(a).
\]
Since the \(E^\times\)-induced character contributes \(0\), we obtain
\[
\chi(g)=-\theta(a)
\]
on the non-semisimple Jordan classes.

\par Finally suppose that $g=\iota(\alpha)$
for some $\alpha\in E^\times\setminus \F_q^\times.
$ Note that \(g\) is not conjugate to any element of \(ZN\), because every
element of \(ZN\) has a repeated eigenvalue in \(\F_q\), whereas \(g\) has
irreducible characteristic polynomial.  Therefore
\[
\chi_{\Ind_{ZN}^{G_q}\eta}(g)=0.
\]
It suffices to compute $\chi_{\Ind_{E^\times}^{G_q}\theta}(g)$. By the induced character formula,
\[
\chi_{\Ind_{E^\times}^{G_q}\theta}(g)
=
\frac{1}{|E^\times|}
\sum_{\substack{x\in G_q\\ x^{-1}gx\in E^\times}}
\theta(x^{-1}gx).
\]
We need to understand the elements \(x\in G_q\) such that
\(x^{-1}gx\in E^\times\). Moreover,
if \(x^{-1}gx\in E^\times\), then it has the same characteristic
polynomial as \(g\) with roots \(\alpha\) and \(\alpha^q\). Thus we have that $x^{-1}gx
\in
\{\iota(\alpha),\iota(\alpha^q)\}$. First suppose that \(x^{-1}gx=\iota(\alpha)=g\). The centralizer of $g$ is $E^\times$ since $g$ is regular semisimple. Thus there are \(|E^\times|\) such elements, and their total contribution is $|E^\times|\theta(\alpha)$.

Next consider the elements \(x\in G\) such that
\[
x^{-1}gx=\iota(\alpha^q).
\]
Choose one \(\F_q\)-linear automorphism $ y\longmapsto y^q$ of \(E\) representing
the Frobenius map and identify it with an element \(\sigma\in G_q\). Note that $\sigma^{-1}=\sigma$. We have that
\[
\sigma^{-1}\iota(\alpha)\sigma
=
\iota(\alpha^q).
\]
Now, $x^{-1}gx= \iota(\alpha^q)$ if and only if
\[x^{-1}gx= \sigma^{-1}\iota(\alpha)\sigma \iff (x\sigma^{-1})^{-1}g(x\sigma^{-1})=\iota(\alpha)=g.\]
It follows that all solutions to \(x^{-1}gx=\iota(\alpha^q)\) form the coset \(E^\times\sigma\).  Once again, there are 
\(|E^\times|\) such elements, and their total contribution is
\[
|E^\times|\theta(\alpha^q).
\]
Therefore
\[
\sum_{\substack{x\in G\\ x^{-1}gx\in E^\times}}
\theta(x^{-1}gx)
=
|E^\times|\theta(\alpha)
+
|E^\times|\theta(\alpha^q).
\]
Dividing by \(|E^\times|\), we get
\[
\chi_{\Ind_{E^\times}^{G}\theta}(g)
=
\theta(\alpha)+\theta(\alpha^q).
\]
Since the \(ZN\)-induced character contributes \(0\), we conclude that
\[
\chi_\theta(g)
=
0-
\bigl(\theta(\alpha)+\theta(\alpha^q)\bigr)
=
-\theta(\alpha)-\theta(\alpha^q).
\]
This completes the proof.
\end{proof}
\section{Density results for average divisibility}
\par In this section, we shall prove the main result of the article. The first large nonzero block comes from evaluating principal series
representations on split regular semisimple conjugacy classes. In this
block, the character values have the form
\[
\chi_1(a)\chi_2(b)+\chi_1(b)\chi_2(a),
\]
which is a sum of two roots of unity.

\par Given a real valued function $f(q)$ and a positive function $g(q)$, we write 
\[f(q)=O(g(q))\] to mean that $\frac{|f(q)|}{g(q)}$ is bounded. Let $\epsilon$ be a positive real number and $g_\epsilon(q)$ be a positive function which depends on $\epsilon$. We write 
\[f(q)=O_\epsilon(g_\epsilon(q))\] to mean that there is a constant $C_\epsilon>0$ which depends on $\epsilon$ such that 
\[\frac{|f(q)|}{g_\epsilon(q)}<C_\epsilon\]for all values of $q$. 

\par The following lemma shows that the
entries in this block which are divisible by \(\ell\) are rare enough to
contribute only to the error term.

\begin{lemma}\label{lemma quad count}
    Let $\epsilon>0$. Then the number of ordered quadruples
\[
(\chi_1,\chi_2,a,b)\in
\left(\widehat{\F_q^\times}\right)^2\times(\F_q^\times)^2
\]
with \(\chi_1\neq \chi_2\) and \(a\neq b\), for which $\ell$ divides
\[
\chi_1(a)\chi_2(b)+\chi_1(b)\chi_2(a),
\]
is \(O_\epsilon(q^{3+\epsilon})\).
\end{lemma}

\begin{proof}
    Setting
\[
\rho:=\chi_1\chi_2^{-1}
\qquad\text{and}\qquad
t:=b/a,
\]
we find that
\[
\chi_1(a)\chi_2(b)+\chi_1(b)\chi_2(a)
=
\chi_1(a)\chi_2(b)\left(1+\frac{\chi_1(b)\chi_2(a)}
{\chi_1(a)\chi_2(b)}\right)
=
\chi_1(a)\chi_2(b)\bigl(1+\rho(t)\bigr).
\]
Since \(\chi_1(a)\chi_2(b)\) is a root of unity, it is a unit in the ring
of algebraic integers.  Therefore divisibility of \(\chi_\pi(g)\) by
\(\ell\) is equivalent to divisibility of $1+\rho(t)$ by \(\ell\). By Lemma \ref{lem:root-unity}, \(1+\rho(t)\) is divisible by \(\ell\) if and only if either:
\begin{itemize}
    \item $\rho(t)=-1$, or,
    \item \(\ell=2\) and $\rho(t)=1$.
\end{itemize}
We show that these cases occur for at most $O_\epsilon(q^{3+\epsilon})$ quadruples $(\chi_1, \chi_2, a,b)$.
\par We set $X:=\widehat{\F_q^\times}$. There are $(q-1)(q-2)$ ordered character pairs \((\chi_1,\chi_2)\) with \(\chi_1\neq\chi_2\) and $(q-1)(q-2)$ ordered element pairs \((a,b)\) with \(a\neq b\). For each fixed nontrivial character
\(\rho\in X\), there are exactly \(q-1\) ordered pairs
\((\chi_1,\chi_2)\) with $\chi_1\chi_2^{-1}=\rho$.  Likewise, for each fixed
\(t\in\F_q^\times\setminus\{1\}\), there are exactly \(q-1\) ordered pairs
\((a,b)\) with $b/a=t$. We call a pair
\[
(\rho,t)\in
\bigl(X\setminus\{1\}\bigr)
\times
\bigl(\F_q^\times\setminus\{1\}\bigr)
\]
exceptional if \(1+\rho(t)\) is divisible by \(\ell\). It remains to count the number of exceptional pairs.
Let \(d\) be the order of \(\rho\).  Then \(d>1\) and
\(d\mid (q-1)\).  The image of \(\rho\) is the cyclic group of \(d\)-th roots
of unity, and every value in the image has exactly $\frac{q-1}{d}$ preimages in \(\F_q^\times\).

First consider the condition \(\rho(t)=-1\).  This has solutions if and
only if \(-1\) belongs to the image of \(\rho\), equivalently if \(d\) is
even. The number of solutions is exactly $\frac{q-1}{d}$. If \(d\) is odd, there are no solutions.  In either case, the number of
solutions to \(\rho(t)=-1\) is at most $\frac{q-1}{d}$. Next consider the condition \(\rho(t)=1\), which only matters when
\(\ell=2\). The kernel of \(\rho\) has size $\frac{q-1}{d}$. Therefore, for a fixed character \(\rho\) of order \(d\), the number of
exceptional \(t\)'s is at most $2\left(\frac{q-1}{d}\right)$.
\par There are \(\varphi(d)\) characters $\rho$ of order \(d\) in
\(X\).  Hence the number of exceptional pairs
\((\rho,t)\) is at most
\[
2(q-1)\sum_{\substack{d\mid q-1\\ d>1}}\frac{\varphi(d)}{d}.
\]
Since \(\varphi(d)/d\leq 1\), this is bounded by
\[
2(q-1)\tau(q-1),
\]
where \(\tau(n)\) denotes the number of positive divisors of \(n\). It is a standard fact (see for example \cite[p.~10]{Murty}) that $\tau(n)=O_\epsilon(n^\epsilon)$. In particular,
\[\tau(q-1)=O_\epsilon(q^\epsilon).\]

\par Now we count the number of quadruples $(\chi_1, \chi_2, a, b)$ such that $\ell$ divides $1+\rho(t)$, where $\rho:=\chi_1\chi_2^{-1}$ and $t:=b/a$. By earlier remarks, each \((\rho,t)\)
arises from \((q-1)^2\) ordered quadruples
\((\chi_1,\chi_2,a,b)\).  Hence their count is at most
\[
2(q-1)^3\tau(q-1)=O_\epsilon(q^{3+\epsilon}).
\]
\end{proof}
\par We next analyze the second large nonzero block, namely the cuspidal
representations evaluated on elliptic conjugacy classes. The character
values in this block are controlled by expressions of the form
\[
-\theta(\alpha)-\theta(\alpha^q).
\]
As in the principal-series case, divisibility by \(\ell\) reduces to a
root-of-unity condition.  The relevant cyclic group is now
\(E^\times/\F_q^\times\), which has order \(q+1\).
\begin{lemma}\label{theta alpha count}
    Let $\epsilon>0$. Then the total number of pairs $(\theta, \alpha)$ where $\theta:E^\times \rightarrow \mathbb{C}^\times$ is a regular character and $\alpha\in E\setminus \F_q$, such that $\ell$ divides 
    \[\theta(\alpha)+\theta(\alpha^q)\] is at most $O_\epsilon(q^{3+\epsilon})$.
\end{lemma}
\begin{proof}
We write
\[
\theta(\alpha)+\theta(\alpha^q)
=
\theta(\alpha)
\left(1+\frac{\theta(\alpha^q)}{\theta(\alpha)}\right).
\]
Setting $\delta_\theta:=\theta^{q-1}$, we have that
\[
\theta(\alpha)+\theta(\alpha^q)
=\theta(\alpha)\bigl(1+\delta_\theta(\alpha)\bigr).
\]
The element \(\theta(\alpha)\) is a unit in the ring of algebraic integers. Thus divisibility by \(\ell\) is equivalent to divisibility of
\(1+\delta_\theta(\alpha)\) by \(\ell\). The character \(\delta_\theta\) is trivial on \(\F_q^\times\) and therefore factors through the quotient $Q:=E^\times/\F_q^\times$. This quotient is cyclic of order
\[
\frac{|E^\times|}{|\F_q^\times|}
=
\frac{q^2-1}{q-1}
=
q+1.
\]
Since \(\theta\) is regular, $\delta_\theta\neq 1$. Consider the map $F:\widehat{E^\times}\rightarrow \widehat Q$ which is defined by \[F(\theta):=\delta_\theta.\] The kernel of $F$ consists exactly
of the Frobenius-invariant characters, i.e., the characters $\theta$ satisfying
\(\theta^q=\theta\): there are $q-1$ such elements. Since \(E^\times\) has \(q^2-1\) characters,
the image of $F$ has size \(q+1\). This is precisely the number of characters of
the quotient \(Q=E^\times/\F_q^\times\).  Thus every character of \(Q\)
occurs as \(\delta_\theta\), and each occurs for exactly \(q-1\) choices of
\(\theta\). 

A cuspidal representation is determined by the
Frobenius orbit \(\{\theta,\theta^q\}\) of a regular character $\theta$, and an
elliptic conjugacy class is determined by the Frobenius orbit
\(\{\alpha,\alpha^q\}\), with
\(\alpha\in E^\times\setminus\F_q^\times\).  Both orbits have size \(2\).
Hence, if we count ordered pairs \((\theta,\alpha)\), then each entry
in the cuspidal/elliptic block is counted exactly four times. 

\par We now count the pairs $(\theta,\alpha)$ such that $1+\delta_\theta(\alpha)$ is divisible by $\ell$. Since \(\delta_\theta\) is trivial on \(\F_q^\times\), the value
\(\delta_\theta(\alpha)\) depends only on the coset
\(u=\alpha\F_q^\times\in Q\).  The condition
\(\alpha\notin\F_q^\times\) is equivalent to \(u\neq 1\).  Therefore the
ordered data relevant to the divisibility question reduce to pairs
\[
(\delta,u)\in
\bigl(\widehat Q\setminus\{1\}\bigr)\times \bigl(Q\setminus\{1\}\bigr).
\]
We call such a pair exceptional if \(1+\delta(u)\) is divisible by
\(\ell\).  By Lemma \ref{lem:root-unity}, this can happen only when:
\begin{itemize}
    \item \(\delta(u)=-1\), or, 
    \item \(\ell=2\) and
\(\delta(u)=1\).
\end{itemize}

Let \(d\) be the order of the nontrivial character \(\delta\).  Then
\(d\mid q+1\).  Since \(Q\) is cyclic of order \(q+1\), every value in the
image of \(\delta\) has exactly \((q+1)/d\) preimages in \(Q\).  Thus
\(\delta(u)=-1\) has no solutions unless \(d\) is even, and in all cases it
has at most \((q+1)/d\) solutions.  Similarly, the equation
\(\delta(u)=1\), after imposing \(u\neq 1\), has at most \((q+1)/d\)
solutions.  Hence, for a fixed \(\delta\) of order \(d\), there are at most
\(2(q+1)/d\) exceptional choices of \(u\).

For every $d|q+1$, there are \(\varphi(d)\) characters of order \(d\) in \(\widehat Q\).
The total number of exceptional pairs \((\delta,u)\) is at most
\[
2(q+1)\sum_{\substack{d\mid q+1\\ d>1}}\frac{\varphi(d)}{d}
\leq 2(q+1)\tau(q+1).
\]
Each pair \((\delta,u)\) lifts to exactly \((q-1)^2\) ordered pairs
\((\theta,\alpha)\): there are \(q-1\) choices of \(\theta\) with
\(\delta_\theta=\delta\), and \(q-1\) elements of \(E^\times\) lying above
the coset \(u\).  Hence the number of exceptional ordered pairs
\((\theta,\alpha)\) is at most
\[
2(q-1)^2(q+1)\tau(q+1)=O_\epsilon(q^{3+\epsilon}).
\]
\end{proof}

\par We now assemble the estimates from the two large nonzero blocks together
with the vanishing of the two large off-diagonal blocks.  The principal
series characters vanish on elliptic classes, while the cuspidal characters
vanish on split regular semisimple classes.  These two zero blocks give the
main term for \(M_0(q)\).  On the other hand, Lemmas
\ref{lemma quad count} and \ref{theta alpha count} show that almost all
entries in the two large nonzero blocks are not divisible by \(\ell\), and
therefore give the main term for \(N_\ell(q)\).

\begin{theorem}\label{explicit estimate theorem}
    Let $\ell$ be a prime number and $\epsilon>0$. Then, 
    \[N_\ell(q)=\frac{q^4}{2}+O\left(q^{3+\epsilon}\right)\]
    and 
    \[M_0(q)=\frac{q^4}{2}+O\left(q^{3+\epsilon}\right).\]
\end{theorem}

\begin{proof}
\par We first prove the estimate for $N_\ell(q)$. From Tables \ref{table: Representations and conjugacy classes} and \ref{table: number of conjugacy classes}, one can see that the one-dimensional and Steinberg families together contain \(2(q-1)\) representations.  Since there are \(q^2-1\) conjugacy classes, the total number of entries in rows $1$ and $3$ of Table \ref{table: Representations and conjugacy classes} is \(O(q^3)\). Similarly, the scalar and Jordan classes together give \(2(q-1)\) columns and thus all entries in columns $1$ and $4$ together have density zero. Thus only the principal series and cuspidal rows evaluated on split regular semisimple and elliptic classes can affect the main term, whereas all other entries in Table \ref{table: Representations and conjugacy classes} contribute to the error term.
\par By Lemma \ref{princseriescalc}, principal series characters vanish on elliptic classes, and therefore, this block does not contribute to \(N_\ell(q)\). By Lemma \ref{cuspcalclemma}, cuspidal characters vanish on split regular semisimple classes, hence this block also does not contribute to \(N_\ell(q)\). It remains to analyze principal series characters on split regular semisimple classes and cuspidal characters on elliptic classes.
\par First, we analyze the character values in principal-series rows on the split regular semisimple
columns. Let $\pi=\Ind_B^{G_q}(\chi_1\otimes\chi_2)$ where $\chi_1\neq \chi_2$
and let $g\sim \diag(a,b)$ with $a,b\in \F_q^\times$ and $a\neq b$. By Lemma \ref{princseriescalc}, the relevant character value is
\[
\chi_\pi(g)
=
\chi_1(a)\chi_2(b)+\chi_1(b)\chi_2(a).
\] It follows from Lemma \ref{lemma quad count} that the total number of pairs $(\pi, g)$ such that $\ell$ divides $\chi_\pi(g)$ is at most $O_\epsilon(q^{3+\epsilon})$.

\par We now treat the cuspidal rows on the elliptic columns. Let $E=\F_{q^2}$ and let
\[
\theta:E^\times\to\C^\times
\]
be a regular character.  Let \(\pi_\theta\) be the corresponding cuspidal
representation, as explained in section \ref{s 2.2}. Let $\chi_\theta$ be the character of $\pi_\theta$. If $g\sim \iota(\alpha)$ with $\alpha\in E^\times\setminus \F_q^\times$,
then Lemma \ref{cuspcalclemma} gives
\[
\chi_{\theta}(g)
=
-\left(\theta(\alpha)+\theta(\alpha^q)\right).
\] According to Lemma \ref{theta alpha count}, there are at most $O_\epsilon(q^{3+\epsilon})$ pairs $(\theta, \alpha)$ for which $\ell$ divides $\chi_\theta(g)$.

Combining the preceding estimates, we have that 
\[N_\ell(q)=\frac{q^4}{2}+O_\epsilon(q^{3+\epsilon}).\]
The main contribution to \(M_0(q)\) comes from the two large blocks which
vanish identically.  Namely, principal-series characters vanish on elliptic
classes, and cuspidal characters vanish on split regular semisimple classes.
Thus these two blocks contribute
\[
2\cdot
\frac{(q-1)(q-2)}2
\cdot
\frac{q(q-1)}2
=
\frac{q(q-1)^2(q-2)}2
=
\frac{q^4}{2}+O(q^3)
\]
zero entries.

It remains to check that the number of entries outside of these blocks that vanish is $O_\epsilon(q^{3+\epsilon})$. The small rows and columns contribute \(O(q^3)\) entries in total.
In the principal-series/split and cuspidal/elliptic blocks, a zero entry is
in particular divisible by \(\ell\), so Lemmas \ref{lemma quad count} and
\ref{theta alpha count} show that the number of zero entries in these blocks is
\(O_\epsilon(q^{3+\epsilon})\). Hence, we find that
\[
M_0(q)=\frac{q^4}{2}+O_\epsilon(q^{3+\epsilon}).
\]This completes the proof.
\end{proof}
We now prove Theorem \ref{thm:main}.
\begin{proof}[Proof of Theorem \ref{thm:main}]
Fix \(0<\epsilon<1\). By the preceding theorem, we have
\[
N_\ell(q)=\frac{q^4}{2}+O_\epsilon(q^{3+\epsilon})
\]
as \(q\to\infty\) through odd prime powers. Since
\[
(q^2-1)^2=q^4-2q^2+1=q^4+O(q^2),
\]
it follows that
\[
\mathfrak d_\ell(q)
=
\frac{N_\ell(q)}{(q^2-1)^2}
=
\frac{\frac{q^4}{2}+O_\epsilon(q^{3+\epsilon})}
{q^4+O(q^2)}.
\]
Dividing numerator and denominator by \(q^4\), and using
\(q^{3+\epsilon}=o(q^4)\), we obtain
\[
\lim_{q\to\infty}\mathfrak d_\ell(q)=\frac12.
\]

It remains to compute the limiting proportion among the nonzero entries.
By the same preceding theorem,
\[
M_0(q)=\frac{q^4}{2}+O_\epsilon(q^{3+\epsilon}).
\]
Hence
\[
M_0'(q)
=
(q^2-1)^2-M_0(q)
=
\left(q^4+O(q^2)\right)
-
\left(\frac{q^4}{2}+O_\epsilon(q^{3+\epsilon})\right),
\]
and therefore
\[
M_0'(q)=\frac{q^4}{2}+O_\epsilon(q^{3+\epsilon}).
\]
Consequently
\[
\mathfrak a_\ell(q)
=
\frac{N_\ell(q)}{M_0'(q)}
=
\frac{\frac{q^4}{2}+O_\epsilon(q^{3+\epsilon})}
{\frac{q^4}{2}+O_\epsilon(q^{3+\epsilon})}.
\]
Again dividing by \(q^4\), we get
\[
\lim_{q\to\infty}\mathfrak a_\ell(q)=1.
\]

Finally, since the proportion of entries not divisible by \(\ell\) tends to
\(1/2\), the proportion of entries divisible by \(\ell\) is
\[
1-\mathfrak d_\ell(q),
\]
and hence also tends to \(1/2\). This proves the theorem.
\end{proof}

\section{Angular distribution for nonzero character values}

\par This section is devoted to the proof of Theorem \ref{thm:argument-equidistribution}. The proof will make use of the following variant of Weyl's criterion (see \cite[Chapter~1, Theorem~2.1]{KN}).
\begin{proposition}\label{prop:weyl-criterion}
For every odd prime power $q$, let \(A_q\) be a finite multiset of real numbers in \([0,2\pi]\). Suppose that for every nonzero integer \(n\), one has
\begin{equation}\label{Weyleqn}\lim_{q\rightarrow\infty}\left(\frac{1}{|A_q|}\sum_{\theta\in A_q} e^{in\theta}\right)=0
\end{equation}
as \(q\to\infty\). Then for any closed interval $I=[a,b]\subseteq [0, 2\pi]$ we have that 
\[\lim_{q\rightarrow \infty} \frac{\# \{\theta\in A_q\mid \theta\in I\}}{|A_q|}=\int_I d\theta=\frac{b-a}{2\pi}.\]
\end{proposition}
\begin{proof}
We note that the statement of the result does not follow directly from Weyl's criterion, but the proof is similar. For $\theta\in [0, 2\pi]$ let $\delta_\theta$ be the Dirac measure supported at $\theta$. Let \(\nu_q\) be the probability measure on \([0,2\pi]\) defined by
\[
\nu_q
:=
\frac{1}{|A_q|}
\sum_{\theta\in A_q}\delta_\theta,
\]
where multiplicities are included. Given a set $J\subseteq [0, 2\pi]$, we have that 
\[\nu_q(J)=\frac{|A_q\cap J|}{|A_q|}.\]
\par Let \(m\) denote normalized Lebesgue
measure on \([0,2\pi]\), so that
\[
dm(\theta)=\frac{d\theta}{2\pi}.
\]
By \eqref{prop:weyl-criterion},
\[
\lim_{q\rightarrow\infty}\left(\int_{0}^{2\pi} e^{in\theta}\,d\nu_q(\theta)\right)= 0
\]
for every $n\in \Z\setminus \{0\}$. Since also
\[
\int_{0}^{2\pi} 1\,d\nu_q(\theta)=1=\int_{0}^{2\pi}1\,dm(\theta),
\]
this is equivalent to saying that
\[
\lim_{q\rightarrow\infty}\left(\int_{0}^{2\pi} e^{in\theta}\,d\nu_q(\theta)
\right)=
\int_{0}^{2\pi} e^{in\theta}\,dm(\theta)
\]
for every \(n\in\mathbb Z\).

\par By linearity, the same convergence holds for every trigonometric
polynomial.  We now pass from trigonometric polynomials to arbitrary
continuous functions on the circle by uniform approximation.  Let \(f\) be a
continuous function on $[0,2\pi]$. By the Stone--Weierstrass theorem,
for every \(\epsilon>0\) there exists a trigonometric polynomial \(P\) such
that
\[
\|f-P\|_\infty<\epsilon.
\]
Since \(\nu_q\) and \(m\) are probability measures, we have
\[
\left|\int_0^{2\pi} (f-P)(\theta)\,d\nu_q(\theta)\right|< \epsilon
\quad\text{and}\quad
\left|\int_0^{2\pi} (f-P)(\theta)\,dm(\theta)\right|< \epsilon.
\]
Therefore
\[
\left|
\int_0^{2\pi} f(\theta)\,d\nu_q(\theta)
-
\int_0^{2\pi} f(\theta)\,dm(\theta)
\right|
<
2\epsilon+
\left|
\int_0^{2\pi} P(\theta)\,d\nu_q(\theta)
-
\int_0^{2\pi} P(\theta)\,dm(\theta)
\right|.
\]
The last term tends to \(0\) as \(q\to\infty\), since \(P\) is a
trigonometric polynomial.  Hence
\[
\limsup_{q\to\infty}
\left|
\int_0^{2\pi} f(\theta)\,d\nu_q(\theta)
-
\int_0^{2\pi} f(\theta)\,dm(\theta)
\right|
<2\epsilon.
\]
Since \(\epsilon>0\) is arbitrary, it follows that
\[
\lim_{q\rightarrow\infty}\int_0^{2\pi} f(\theta)\,d\nu_q(\theta)
=
\int_0^{2\pi} f(\theta)\,dm(\theta)
\]
for every continuous function \(f\) on $[0,2\pi]$.  Thus \(\nu_q\)
converges weakly to normalized Lebesgue measure on the circle.

\par Recall that a subset $J$ is called an $m$-continuity set if its boundary has Lebesgue measure $0$. It follows from the Portmanteau theorem (see \cite[Theorem~2.1, (i) implies (v)]{Billingsley}) that 
\[\lim_{q\rightarrow \infty} \nu_q(J)=m(J).\] Since $I$ is an \(m\)-continuity set, we deduce that
\[
\lim_{q\to\infty}\nu_q(I)=m(I)=\frac{b-a}{2\pi}.
\]This completes the proof.
\end{proof}
Let $q$ be an odd prime and $\Irr(G_q)$ be the set of isomorphism classes of complex irreducible representations of $G_q$. The group of characters $X_q:=\widehat{\F_q^\times}$ acts on \(\Irr(G_q)\) by twisting:
\[
\eta(\pi):=\pi\otimes(\eta\circ\det),
\]
where $\eta\in X_q$ and $\pi\in \Irr(G_q)$. 

\begin{lemma}\label{lem:twisting-cuspidal-parameters}
Let \(E=\F_{q^2}\), let \(\theta:E^\times\to \C^\times\) be a regular
character and $\pi_\theta$ denote the associated cuspidal representation. Let \(\eta:\F_q^\times\to\C^\times\) be a character.  Then
\[
\pi_\theta\otimes(\eta\circ\det)
\simeq
\pi_{\theta\cdot(\eta\circ N_{E/\F_q})}.
\]
\end{lemma}

\begin{proof}
We first recall a general elementary fact about induction.  Let
\(H\leq G_q\), let \(\sigma\) be a representation of \(H\), and let
\(\lambda\) be a one-dimensional representation of \(G_q\).  Then
\[
\Ind_H^{G_q}(\sigma)\otimes \lambda
\simeq
\Ind_H^{G_q}\bigl(\sigma\otimes \lambda|_H\bigr),
\]
see \cite[Exercise 3.16, p.~34]{fultonharris}. We apply this with \(\lambda=\eta\circ\det\).  Recall that the cuspidal
representation attached to \(\theta\) is constructed as the virtual
representation
\[
\pi_\theta
=
\Ind_{ZN}^{G_q}\varphi_\theta
-
\Ind_{E^\times}^{G_q}\theta,
\]
where
\[
\varphi_\theta(aIu_x)=\theta(a)\psi(x),
\qquad a\in\F_q^\times,\ x\in\F_q, \quad \text{ and } \quad u_x=
\begin{pmatrix}
1 & x\\
0 & 1
\end{pmatrix}.
\]
We claim that twisting both induced terms by \(\eta\circ\det\) replaces
\(\theta\) by
\[
\theta'=\theta\cdot(\eta\circ N_{E/\F_q}).
\]

First consider the \(E^\times\)-term.  For \(\alpha\in E^\times\), viewed
as an element of \(G_q\) via the embedding \(\iota:E^\times\hookrightarrow
G_q\), one has
\[
\det(\iota(\alpha))=N_{E/\F_q}(\alpha).
\]
Therefore
\[
(\eta\circ\det)|_{E^\times}
=
\eta\circ N_{E/\F_q}.
\]
By the induction-tensor compatibility recalled above,
\[
\Ind_{E^\times}^{G_q}\theta\otimes(\eta\circ\det)
\simeq
\Ind_{E^\times}^{G_q}
\bigl(\theta\cdot(\eta\circ N_{E/\F_q})\bigr)
=
\Ind_{E^\times}^{G_q}\theta'.
\]

Next consider the \(ZN\)-term.  For \(a\in\F_q^\times\) and \(u_x\in N\),
we have
\[
\det(aIu_x)=a^2.
\]
On the other hand, regarding \(a\) as an element of
\(\F_q^\times\subset E^\times\), one has
\[
N_{E/\F_q}(a)=a^{q+1}=a^2,
\]
since \(a^q=a\).  Thus
\[
(\eta\circ\det)(aIu_x)
=
\eta(a^2)
=
(\eta\circ N_{E/\F_q})(a).
\]
Hence
\[
\varphi_\theta(aIu_x)\,(\eta\circ\det)(aIu_x)
=
\theta(a)\psi(x)\eta(a^2)
=
\theta'(a)\psi(x).
\]
This is precisely the character \(\varphi_{\theta'}\) of \(ZN\).  Therefore
\[
\Ind_{ZN}^{G_q}\varphi_\theta\otimes(\eta\circ\det)
\simeq
\Ind_{ZN}^{G_q}\varphi_{\theta'}.
\]

Combining the two terms in the virtual construction of \(\pi_\theta\), we
obtain
\[
\pi_\theta\otimes(\eta\circ\det)
\simeq
\Ind_{ZN}^{G_q}\varphi_{\theta'}
-
\Ind_{E^\times}^{G_q}\theta'
=
\pi_{\theta'}.
\]
Since \(\theta'=\theta\cdot(\eta\circ N_{E/\F_q})\), this proves the lemma.
\end{proof}

Denote by $\op{Stab}(\pi)$ the stabilizer of $\pi$, consisting of $\eta\in X_q$ such that $\eta(\pi)\simeq \pi$.

\begin{lemma}\label{mainlemmalastsection}
    Let \(\mathcal B_q\subset
\Irr(G_q)\) be the set of the one-dimensional representations,
the Steinberg twists, and those principal-series or cuspidal representations $\pi$ for which $\op{Stab}(\pi)$ is nontrivial. Then we have that \[
|\mathcal B_q|=O(q).
\]
\end{lemma}

\begin{proof}
The one-dimensional representations and the Steinberg twists together contribute
\(2(q-1)\) rows. We now count the remaining rows with nontrivial stabilizer. 
\par Let $\pi=\Ind_B^{G_q}(\chi_1\otimes\chi_2)$ with $\chi_1\neq\chi_2$ be a principal-series representation. Suppose that \(\pi\) has nontrivial
stabilizer under twisting.  Then there exists a nontrivial character
\(\eta\in X_q\) such that
\[
\Ind_B^{G_q}(\chi_1\eta\otimes\chi_2\eta)
\simeq
\Ind_B^{G_q}(\chi_1\otimes\chi_2).
\]
Using the standard equivalence relation for principal series, this implies
\[
\{\chi_1\eta,\chi_2\eta\}=\{\chi_1,\chi_2\}.
\]
Since \(\eta\neq 1\), the equality cannot hold with
\(\chi_1\eta=\chi_1\) and \(\chi_2\eta=\chi_2\). Therefore one must have $\chi_1\eta=\chi_2$
and $\chi_2\eta=\chi_1$. It follows that \(\eta^2=1\). Since \(q\) is odd, there is a unique
nontrivial quadratic character $\eta_2$ of \(\F_q^\times\). Thus the principal-series
rows with nontrivial stabilizer are precisely of the form $\Ind_B^{G_q}(\chi\otimes \chi\eta_2)$
where \(\eta_2\) is the quadratic character. There are at most \(q-1\) such representations.
\par We next consider the cuspidal representation with nontrivial stabilizer. Let $E=\F_{q^2}$, and let
\(\pi_\theta\) be the cuspidal representation attached to a regular
character $\theta:E^\times\longrightarrow \C^\times$. Assume that there exists a nontrivial character $\eta\in X_q$ such that $\eta(\pi_\theta)\simeq \pi_\theta$. By Lemma \ref{lem:twisting-cuspidal-parameters} we have that:
\[
\pi_\theta\otimes(\eta\circ\det)
\simeq
\pi_{\theta\cdot(\eta\circ N_{E/\F_q})},
\]
and therefore,
\[
\pi_{\theta\cdot(\eta\circ N_{E/\F_q})}\simeq \pi_\theta.
\]
This means that
\[
\theta\cdot(\eta\circ N_{E/\F_q})=\theta
\quad\text{or}\quad
\theta\cdot(\eta\circ N_{E/\F_q})=\theta^q.
\]
The first equality gives \(\eta\circ N_{E/\F_q}=1\).  Since the norm map
\(N_{E/\F_q}:E^\times\to\F_q^\times\) is surjective, this implies
\(\eta=1\), contrary to our assumption.  Thus the second equality must hold:
\[
\eta\circ N_{E/\F_q}=\theta^q\theta^{-1}.
\]
Restricting this equality to \(\F_q^\times\subset E^\times\), we obtain
\[
\eta(a^2)=1
\qquad\text{for all }a\in\F_q^\times,
\]
because \(\theta^q(a)=\theta(a)\) for \(a\in\F_q^\times\).  It follows that \(\eta^2=1\), and therefore \(\eta\) must be the unique quadratic character \(\eta_2\) of $\F_q^\times$.

It remains to count the possible \(\theta\) satisfying $\theta^q\theta^{-1}=\eta_2\circ N_{E/\F_q}$.  Consider the homomorphism
\[
\widehat{E^\times}\longrightarrow \widehat{E^\times},
\qquad
\theta\longmapsto \theta^{q-1}.
\]Let $\theta_1$ and $\theta_2$ be characters of $E^\times$ such that
\[\theta_1^{q-1}=\theta_2^{q-1}=\eta_2\circ N_{E/\F_q}.\] Then $\alpha:=\theta_1\theta_2^{-1}$ is in the kernel of the above homomorphism. This kernel consists exactly of the Frobenius-invariant characters and has size \(q-1\).
Suppose that $\theta_1$ is a solution to 
$\theta_1^{q-1}=\eta_2\circ N_{E/\F_q}$, then any other solution is of the form $\theta_1\alpha$. Since there are at most $q-1$ choices for $\alpha$ there are consequently at most $q-1$ choices for $\theta$
satisfying
$
\theta^q\theta^{-1}=\eta_2\circ N_{E/\F_q}.$
Passing from regular characters \(\theta\) to cuspidal representations
\(\pi_\theta\) can only decrease this number.  Hence there are at most $O(q)$ cuspidal representations with
nontrivial stabilizer. This proves that
\[
|\mathcal B_q|=O(q).
\]
\end{proof}

\begin{proof}[Proof of Theorem \ref{thm:argument-equidistribution}]
Let \(\mathscr C_q\) denote the set of conjugacy classes of \(G_q\). For a complex number \(z\), write
\[
u(z):=\begin{cases}
    e^{i\op{arg}(z)}&\text{ if } z\neq 0;\\
    0&\text{ if }z=0.\\
\end{cases}
\]
Let \(A_q\) be the multiset of arguments
\[
A_q
:=
\left\{
\arg(\chi_\pi(C)) \mid
\pi\in\Irr(G_q),\ C\in\mathscr C_q,\ \chi_\pi(C)\neq 0
\right\},
\]
where each pair \((\pi,C)\) with \(\chi_\pi(C)\neq 0\) is counted once.  Thus
\[
|A_q|=M_0'(q),
\]
the number of nonzero entries in the character table of $G_q$. By Theorem \ref{explicit estimate theorem},
\[
M_0'(q)=\frac{q^4}{2}+O_\epsilon(q^{3+\epsilon}).
\] 

\par By Proposition \ref{prop:weyl-criterion}, it is enough to show that for every \(n\in\Z\setminus\{0\}\),
\begin{equation}\label{Weyllimit}\lim_{q\rightarrow \infty}\left(\frac{1}{M_0'(q)}\sum_{\substack{\pi\in \Irr(G_q)\\ C\in\mathscr C_q}}u(\chi_\pi(C))^n\right)=0.\end{equation}
Fixing $n$, we prove the stronger estimate
\begin{equation}\label{Sqeqn}
S_q(n):=
\left|\sum_{\substack{\pi\in \Irr(G_q)\\ C\in\mathscr C_q}}
u(\chi_\pi(C))^n\right|\leq C_0|n|q^3,
\end{equation}
where $C_0>0$ is an absolute constant.
Since \(M_0'(q)\sim q^4/2\) and $n$ is fixed, the limit \eqref{Weyllimit} must vanish.
\par If \(C\in\mathscr C_q\), then \(\det(g)\) is independent of the choice of
\(g\in C\); we denote this common value by \(\det C\).  For every
\(\pi\in\Irr(G_q)\), every \(\eta\in X_q\), and every \(C\in\mathscr C_q\),
one has
\[
\chi_{\pi\otimes(\eta\circ\det)}(C)
=
\eta(\det C)\chi_\pi(C),
\]
and therefore,
\begin{equation}\label{twistingreln}u\left(\chi_{\pi\otimes(\eta\circ\det)}(C)\right)^n
=
\eta(\det C)^n u(\chi_\pi(C))^n.
\end{equation}
\par We recall from the statement of Lemma \ref{mainlemmalastsection} that $\mathcal{B}_q$ consists of one-dimensional representations,
the Steinberg twists, and those principal-series or cuspidal representations $\pi$ for which $\op{Stab}(\pi)$ is nontrivial. Since each term
\(u(\chi_\pi(C))^n\) has absolute value $\leq 1$, we have that
\[
\left|\sum_{\substack{\pi\in \mathcal{B}_q\\ C\in\mathscr C_q}}
u(\chi_\pi(C))^n\right|\leq |\mathcal{B}_q||\mathscr{C}_q|. \]
Note that \(|\mathscr C_q|=q^2-1\) and by Lemma \ref{mainlemmalastsection}, $|\mathcal B_q|=O(q)$. Therefore, we find that 
\[
|\mathcal{B}_q||\mathscr{C}_q|
=
O(q^3).
\]
%and in particular it follows that 
%\[\lim_{q\rightarrow \infty}\left(\frac{1}{M_0'(q)}\sum_{\substack{\pi\in \mathcal B_q\\ C\in\mathscr C_q}}u(\chi_\pi(C))^n\right)=0.\]
Therefore, \eqref{Sqeqn} follows if we prove that 
\begin{equation}\label{toproveeqn}
\left|\sum_{\substack{\pi\in \mathcal \Irr(G_q)\setminus \mathcal B_q\\ C\in\mathscr C_q}}
u(\chi_\pi(C))^n\right|\leq C_1 |n|q^3
\end{equation}
for some absolute constant $C_1>0$. 

\par Since representations outside $\mathcal{B}_q$ have trivial stabilizer, $\Irr(G_q)\setminus \mathcal B_q$ decomposes into free \(X_q\)-orbits $\op{Orb}_q=\op{Orb}_{X_q}\left(\Irr(G_q)\setminus \mathcal B_q\right)$. For each orbit $\cO$, fix a representation $\pi_{\cO}\in \cO$. From \eqref{twistingreln}, we deduce that
\begin{equation}\label{abovesum}\sum_{\substack{\pi\in \mathcal \Irr(G_q)\setminus \mathcal B_q\\ C\in\mathscr C_q}}
u(\chi_\pi(C))^n=\sum_{\mathcal O\in \op{Orb}_q}\sum_{C\in \mathscr{C}_q} \left(\sum_{\eta\in X_q} \eta(\det C)^n\right)
u(\chi_{\pi_\cO}(C))^n.\end{equation}
We note that $|X_q|=q-1$. By the orthogonality of characters of the finite abelian group
\(\F_q^\times\),
\begin{equation}\label{orthrelns}\sum_{\eta\in X_q}\eta(\det C)^n
=
\begin{cases}
q-1, & (\det C)^n=1,\\
0, & (\det C)^n\neq 1.
\end{cases}
\end{equation}
Consequently, a free orbit can contribute in the column \(C\) only if $(\det C)^n=1$.
Let $\mu_{|n|}$ be the $|n|$-th roots of unity in $\F_q$. Observe that \[|\op{Orb}_q|\leq \frac{|\op{Irr}(G_q)|}{|X_q|}=\frac{q^2-1}{q-1}=q+1.\]From \eqref{abovesum} and \eqref{orthrelns}, we find that
\[\begin{split}\left|\sum_{\substack{\pi\in \mathcal \Irr(G_q)\setminus \mathcal B_q\\ C\in\mathscr C_q}}
u(\chi_\pi(C))^n\right|
\leq &(q-1)\sum_{\xi\in \mu_{|n|}}\sum_{\cO\in \op{Orb}_q} \sum_{\substack{C\in \mathscr{C}_q\\ \det C=\xi}}
|u(\chi_{\pi_\cO}(C))|^n\\
\leq & (q-1)\sum_{\xi\in \mu_{|n|}}\sum_{\cO\in \op{Orb}_q} \sum_{\substack{C\in \mathscr{C}_q\\ \det C=\xi}}1\\
=&(q-1)|\op{Orb}_q|\sum_{\xi\in \mu_{|n|}}\#\{C\in \mathscr{C}_q\mid \det C=\xi\}\\
\leq &(q^2-1)\sum_{\substack{\xi\in \mu_{|n|}\\t\in \F_q}}\#\{C\in \mathscr{C}_q\mid \op{trace}C=t,\quad \det C=\xi\}\\
\leq &|n|q(q^2-1)\op{max}_{\substack{\xi\in \mu_{|n|}\\t\in \F_q}}\#\{C\in \mathscr{C}_q\mid \op{trace}C=t,\quad \det C=\xi\}.
\end{split}\]
For fixed $d\in \F_q^\times$ and $t\in \F_q$, we claim that
\[\#\{C\in \mathscr{C}_q\mid \op{trace}C=t,\quad \det C=d\}\leq 2.\] 
The characteristic polynomial \[\op{ch}_C(X)=X^2-tX+d\] has a repeated root whenever $t^2-4d=0$. First consider the case when $\op{ch}_C(X)$ is irreducible over $\F_q$ and let $\alpha\in \F_{q^2}\setminus \F_q$ be a root. Then $\alpha$ and $\alpha^q$ are the roots of $\op{ch}_C(X)$ and $C$ is conjugacy class of either $\iota(\alpha)$ or $\iota(\alpha^q)$. Thus there are two choices for $C$. Next, assume that $\op{ch}_C(X)$ splits in $\F_q$ and let $a,b\in \F_q$ be its roots up to multiplicity. If $a\neq b$ then $C$ is the conjugacy class of $\op{diag}(a,b)$. On the other hand, if $a=b$ then $C$ is the conjugacy class of $\mtx{a}{0}{0}{a}$ or $\mtx{a}{1}{0}{a}$. Thus in either case, there are at most $2$ choices for $C$. This proves the claim. 

Therefore, we have shown that 
\[\left|\sum_{\substack{\pi\in \mathcal \Irr(G_q)\setminus \mathcal B_q\\ C\in\mathscr C_q}}
u(\chi_\pi(C))^n\right|\leq 2 |n|q^3,\]
thus proving \eqref{toproveeqn} and the estimate \eqref{Sqeqn}. This completes the proof of the Theorem.
\end{proof}

\bibliographystyle{alpha}
\bibliography{references}

\end{document}